\documentclass[11pt]{article}

\usepackage{amsmath}
\usepackage{amssymb}
\usepackage{amsthm}
\newtheorem{theorem}{Theorem}[section]
\newtheorem{lemma}[theorem]{Lemma}
\makeatletter \@addtoreset{equation}{section} \makeatother

\newcommand{\beq}{\begin{equation}}
\newcommand{\eeq}{\end{equation}}

\newcount\gnum
\def\G#1{
  \gnum=`#1
  \ifnum \gnum>`Z
     \advance\gnum by-`a
     \ifcase \gnum \alpha \or \beta \or \gamma \or \delta \or
        \epsilon\or\phi\or\theta\or\eta\or\iota\or\or\kappa\or
        \lambda\or\mu\or\nu\or\or\pi\or\xi\or\rho\or\sigma\or
        \tau\or\upsilon\or\or\omega\or\chi\or\psi\or\zeta\fi
  \else
     \advance\gnum by-`A
     \ifcase \gnum \Alpha \or \Beta \or \Gamma \or \Delta \or
        \Epsilon\or\Phi\or\Theta\or\Eta\or\Iota\or\or\Kappa\or
        \Lambda\or\Mu\or\Nu\or\or\Pi\or\Xi\or\Rho\or\Sigma\or
        \Tau\or\Upsilon\or\or\Omega\or\Chi\or\Psi\or\Zeta\fi
  \fi}
\def\vp{\varphi}
\def\half{\TS{\frac{1}{2}}}
\def\R{\mathbb{R}}
\def\TS{\textstyle}
\def\com#1{\quad{\textrm{#1}}\quad}
\def\eq#1{(\ref{#1})}
\def\nn{\nonumber}
\def\ge{\geqslant}
\def\le{\leqslant}

\begin{document}

\title{The Vacuum in Nonisentropic Gas Dynamics}

\author{Geng Chen\thanks{\texttt{chen@math.psu.edu}. Department of
    Mathematics, Pennsylvania State University, University Park, PA,
    16802},\quad 
  Robin Young\thanks{\texttt{young@math.umass.edu}. Department of
    Mathematics, University of Massachusetts, Amherst, MA
    01003. Supported in part by NSF Applied Mathematics Grant Number
    DMS-0908190.}}
\maketitle

\begin{abstract}
  We investigate the vacuum in nonisentropic gas dynamics in one space
  variable, with the most general equation of states allowed by
  thermodynamics.  We recall physical constraints on the equations of
  state and give explicit and easily checkable conditions under which
  vacuums occur in the solution of the Riemann problem.  We then
  present a class of models for which the Riemann problem admits
  unique global solutions without vacuums.
\end{abstract}

2000\textit{\ Mathematical Subject Classification:} 35L65, 35B65,
35B35.

\textit{Key Words:}  Nonisentropic gas dynamics,  conservation laws,
vacuum, Large data, Riemann
problem.
%
%
\section{Introduction}

We consider the Euler equations of fluid dynamics in one space
dimension,
\begin{align}
  \rho_t+(\rho u)_x &=0,\nn\\
  (\rho u)_t+(\rho u^2+p)_x &=0,\label{euler}\\ 
  (\half\rho u^2+\rho E)_t+(u(\half\rho u^2+\rho E+p))_x&=0,\nn
\end{align} 
describing conservation of mass, momentum and energy, respectively.
Here $\rho$ is the density, $p$ is the pressure, $u$ is the velocity,
and $E$ is the specific internal energy of the fluid: the equations
are closed by specifying an equation of state or constitutive
relation, describing how the thermodynamic variables are related.
This equation of state depends on the molecular structure of the fluid
and is subject to physical constraints such as the Second Law of
Thermodynamics.

The Cauchy problem for \eq{euler} is well understood when the initial
data have small total variation~\cite{Glimm,Bressan,smoller,Dafermos},
but little is known for solutions with large data~\cite{TY,TYperStr}.
One of the main difficulties is the possible occurrence of vacuum.
Several authors have studied the vacuum in isentropic gas dynamics
(obtained by dropping the energy equation)~\cite{li tatsien, liu,
  smith, Young p system}, and more recently nonisentropic polytropic
ideal ($\G c$-law) gases~\cite{chen_jenssen, chen_young1,
  chen_young2}.

Our goal here is to describe in detail which gases admit vacuums in
the solution.  It is well known that $\G c$-law gases require vacuum
in order to solve the global Riemann problem~\cite{smoller}, whereas
an isothermal gas ($\G c=1$) does not.  Also, a better existence
theory is available for the isothermal system, due to a degenerate
wave curve structure~\cite{Nishida}, but there is no natural
$3\times3$ physically consistent analogue of the isothermal system.
In \cite{temple}, Temple considers a class of $3\times3$ constitutive
relations with simplified structure, but this system violates some
thermodynamic constraints, given in \cite{lieb,menikoff}.  Our
intention is to present a class of physically consistent constitutive
laws which do not admit vacuum, thus removing the issue of vacuum for
these equations of state, in order to focus more fully on the effects
of nonlinear wave interactions.

We begin by collecting all the physical conditions that restrict the
equation of state.  We recall the solution of the Riemann problem and
give a necessary and sufficient condition for the occurrence of
vacuums in the general solution of the Riemann problem.  By exhibiting
specific examples, we describe a class of constitutive laws which
satisfy all our physical constraints and for which the Riemann problem
does not contain a vacuum state.  We then present the simplest class of
such equations of state, namely
\[
  E = \frac{K_0\,e^{S/c_\tau}}{2\sigma-1}(\ln(\tau+1))^{1-2\sigma},
\qquad
  p=K_0\,e^{S/c_\tau}\,\frac{(\ln(\tau+1))^{-2\sigma}}{\tau+1},
\]
where $\half<\G s\leqslant1$.  Although it is implied by quoted results, we
explicitly solve the global Riemann problem for this pressure law,
without use of the vacuum state.  It is expected that this will inform
such a study of wave interactions, which is the subject of the
authors' ongoing research.

%
%

\section{Thermodynamic constraints}

To set notation, we describe the thermodynamic constraints.  The
thermodynamic properties of a fluid are embodied in the constitutive
relation $E=E(\tau,S)$, where $\tau=1/\rho$ is the specific volume and
$S$ is the specific entropy.  See~\cite{menikoff} for a physical discussion
of these constraints, and~\cite{smith} for a detailed mathematical
analysis.

First, we make the \emph{\textbf{smoothness assumption}}, 
\beq
  E=E(\tau,S)\in C^2,\com{for}\tau\in \R^+\com{and} S\in\R,
\label{Esmooth}
\eeq
which is true for most fluids.

We require the fluid to satisfy the Second Law of Thermodynamics,
which asserts that
\beq 
  dE=T\,dS-p\,d\tau,
\label{eos}
\eeq 
where $T$ is the temperature.  This in turn implies that 
\beq 
  E_S=T,\com{and} E_\tau=-p.
\label{e_s_e_tau}
\eeq

We assume the \emph{\textbf{standard thermodynamic constraints}}:
specific volume $\G t$, pressure $p$ and temperature $T$
satisfy
\[
  \tau > 0, \quad p > 0 \com{and} T \ge 0;
\]
so that by \eq{e_s_e_tau},
\[
  E_\tau<0 \com{and} E_S\ge 0.
\]
We assume ``stability of matter'', which asserts that the energy is
finite,
\[
  E_\infty=\lim_{\tau\to\infty}E <\infty,
\]
and without loss of generality we take
\beq
   E_\infty=0, \com{so that} E(\G t,S)>0
\label{e vacuum}
\eeq
for all $\G t>0$ and $S$, see~\cite{lieb}.

Next, we assume the \emph{\textbf{thermodynamic stability
constraint}} that the energy $E=E(\tau,S)$ be jointly convex,
\beq
  E_{\tau\tau}=-p_{\tau}> 0, \quad E_{SS}=T_S\geqslant 0,
\label{p tau}
\eeq
while also
\beq
  E_{\tau\tau}\cdot E_{SS}\geqslant E^2_{\tau S} \com{and}
  E_{\tau S}=-p_S\leqslant 0.
\label{E cvx}
\eeq
According to \cite{menikoff}, thermodynamic stability yields \eq{p tau}
and \eq{E cvx}, but our discussion requires only \eq{p tau}.  Equation
\eq{p tau} in turn implies that the system is strictly hyperbolic away
from vacuum.  The condition $p_S>0$ states that the material expands
upon heating at constant pressure.  We assume nonstrict inequality for
$p_S$ to include isentropic gas dynamics, for which $p_S\equiv0$.

Our final condition is an \emph{\textbf{energy condition}}, which
states that if the pressure $p(\G t,S)$ is specified, then the energy
$E$  is well defined: that is,
\[
  E(\tau,S)=\int_\tau^{\infty}p(\tau',S)\;d\tau',
\]
where we have used (\ref{e_s_e_tau}) and (\ref{e vacuum}).  That is, we
require that, for all $S$,
\beq
  \int_1^{\infty}p(\tau,S)\;d\tau=\int_0^{1}\frac{\hat{p}(\rho,S)}
   {\rho^2}\;d\rho<+\infty,
\label{energy condition}
\eeq
where $\hat p$ is defined by
\beq
  \hat p(\G r,S)\equiv p(1/\rho,S) = p(\G t,S).
\label{phat}
\eeq

The energy condition \eq{energy condition} imposes growth conditions
on $p(\G t,S)$, or equivalently restricts the pressure $\hat p$ near
vacuum, namely $\hat p(0+,S)=0$, and by l\'{}Hospital's rule, 
\beq
  \lim_{\G r\to 0} \hat p_\rho(\G r,S) = 
  \lim_{\G r\to 0} \frac{\hat{p}(\rho,S)}{\rho} = 0.
\label{p0}
\eeq

Note that our conditions alone are \emph{not} sufficient to conclude
uniqueness of solutions to the Riemann problem: uniqueness is assured
if and only if Smith's \emph{medium condition}, that is
\beq
  \frac{\partial}{\partial \tau}p(\tau,E)\le \frac{p^2}{2E},
\label{medium}
\eeq 
where $p$ is regarded as $p(\G t,E)$, is satisfied, see~\cite{smith}.

%
%
\section{Vacuum in the solution of Riemann problems}

We wish to investigate the circumstances in which a vacuum appears in
the solution to a Riemann problem.  We briefly recall the solution of
the Riemann problem; see~\cite{smoller}.

\subsection{Riemann problem}

We begin by calculating the simple (rarefaction) wave curves.  For
smooth solutions, we replace the third (energy) equation of
(\ref{euler}) by the entropy equation
\[
  S_t+uS_x=0,
\]
and use $(\G r, u, S)$ as the state variables.  It is routine to
calculate the eigensystem after writing \eq{euler} in
quasilinear form.  The eigenvalues of (\ref{euler}) are
\[ 
  \lambda_1=u-c,\quad \lambda_2=u,\quad \lambda_3=u+c,
\]
and these are the wavespeeds of the
backward, middle and forward waves, respectively, and
\beq
  c = c(\G r,S) :=\sqrt{\hat{p}_\rho}
\label{c_def}
\eeq 
is the speed of sound.  As is well known, the forward and backward
waves are genuinely nonlinear and the middle waves linearly
degenerate.   The corresponding
eigenvectors are
\[
r_1=\left (\begin{array}{l}
 \rho\\
 -c\\
 0
 \end{array}\right ),\quad
  r_2=\left (\begin{array}{l}
 \hat{p}_S\\
 0\\
 -\hat{p}_\rho
 \end{array}\right ),\quad
 r_3=\left (\begin{array}{l}
 \rho\\
 c\\
 0
\end{array}\right ).
\]
It follows that the equation of a backward simple wave is
\beq 
  u-u_l = R(\G r_l,S) - R(\G r,S), \quad S=S_l,
\label{1 simple wave}
\eeq
where the subscript $l$ refers to the left state of the wave, and we
define
\beq
  R(\G r,S) = \int_1^{\rho}\frac{c(r,S)}{r}\;dr
       = \int_1^{\rho}\frac{\sqrt{\hat p_\rho}}{r}\;dr,
\com{for} \G r \ge 0.
\label{Rdef}
\eeq 
The equation of a forward simple wave curve is 
\beq 
  u_r-u = R(\G r_r,S) - R(\G r,S), \quad S=S_r,
\label{3 simple wave}
\eeq 
where the subscript $r$ refers to the right state of the wave. 

Next we calculate the shock curves: these are described by the
Rankine-Hugoniot conditions,
\begin{align}
  \xi[\rho]&=[\rho u],\nn\\
  \xi[\rho u]&=[\rho u^2+p],\label{RH}\\ 
  \xi[\frac{1}{2}\rho u^2+\rho E]&=[u(\frac{1}{2}\rho u^2+\rho
E+p)],\nn
\end{align}
where $\xi$ is the shock speed and the brackets denote the jump in a
quantity across the shock.  We simplify these as follows: the first
equation can be written
\beq
  \xi\frac{\rho_r}{\rho_l}-\xi-\frac{\rho_r}{\rho_l}u_r+u_l=0,
\label{RH4}
\eeq
and, recalling that $\G t=1/\G r$, using this in the second equation
and simplifying yields
\[
  [p][\tau]=\xi[\rho u][\tau]-[\rho u^2][\tau]=-[u]^2.
\]
Next, denoting the average of a quantity by
$\overline{g}=\frac{g_l+g_r}{2}$, manipulating the first two equations
of \eq{RH} yields
\beq
  \xi\overline{\rho}[u]=\overline{\rho u}[u]+[p],
\com{so}
  \xi\overline{\rho}-\overline{\rho u}=\frac{[p]}{[u]}.
\label{RH8}
\eeq
  
The third equation of \eq{RH} gives, after simplifying,
\[
  \xi\overline{\rho}[\frac{1}{2}u^2+E]=
  \overline{\rho u} [\frac{1}{2}u^2+E]+[up].
\]
Using (\ref{RH8}) and again simplifying, we finally obtain
\beq
  [E]+\overline{p}[\tau]=0,\label{HC}
\eeq 
which is the \emph{Hugoniot curve} for shocks.  We conclude that
this describes the shock curve fully: first, solve \eq{HC} to find the
relation between $\G r$ and $S$, then use 
\beq
  [u] = - \sqrt{-[p]\,[\G t]},
\label{ueq}
\eeq
obtained from the entropy condition~\cite{lax,courant}, to resolve
$u$, and finally use \eq{RH8} to determine $\xi$.

Recall that an entropy condition is required to choose admissible
shocks and thus obtain uniqueness of Riemann
solutions~\cite{courant,lax}.  This condition states that pressure
(and thus also density) is bigger behind the shock, and leads to the
negative square root in \eq{ueq} above.  It follows similarly that the
density behind a (forward or backward) rarefaction wave is
\emph{smaller} than the density ahead of the wave.  Also, this implies
that the state behind a shock \emph{cannot} be the vacuum state.

It is routine to describe contact discontinuities using \eq{RH}:
namely, substitute $[u]=0$ in directly, to obtain
\beq
   [u] = 0,\quad [p]=0 \com{and} \xi = \overline u = u.
\label{contact}
\eeq

If a contact discontinuity is adjacent to a vacuum, then we combine
the contact discontinuity and the vacuum region into a new vacuum
region, called a {\em{non-isentropic vacuum}} on which the entropy
density $\G rS$ vanishes.  The left and right hand limits of $S$ on
the left and right boundaries of such a vacuum region are different.
It follows that, if the vacuum is involved in the solution of the
Riemann problem, it can only be generated between two outgoing
rarefaction waves, see~\cite{liu,li tatsien,Young p system}

\subsection{Vacuum condition}

\begin{lemma}
\label{lem:R}
The vacuum state exists in the solution of Riemann problems of
(\ref{euler})  if and only if, for some $S$,
\beq
  R(0+,S) > - \infty,
\label{vacuum condition}
\eeq
where $R$ is defined in \eq{Rdef}.
\end{lemma}

\begin{proof}
We first prove that if the vacuum state exists in the solutions of
Riemann problems then \eq{vacuum condition} is satisfied.  Recall that
vacuum state only appears between two rarefaction waves.  If a Riemann
solution consists of forward and backward rarefactions, it follows
that the velocity $u$ is monotone increasing as a function of
$x$~\cite{courant}.  Parameterizing the forward wave by $\rho$, it
follows from \eq{3 simple wave} that
\[
  u_r - u = R(\rho_r,S_r) - R(\rho,S_r);
\]
now since $u\ge u_l$, we get the uniform bound
\[
  R(\rho,S_r) \ge u_l - u_r + R(\rho_r,S_r),
\]
and allowing $\rho\to0$ implies \eq{vacuum condition}.

Now suppose (\ref{vacuum condition}) holds for some $S$.  We claim
that the Riemann problem with data $U_l = (1,0,S)$, and $U_r =
(1,u_r,S)$ has a vacuum in the solution whenever
\[
  u_r > - 2\,R(0+,S).
\]
To see this, assume $u_r>0$ and resolve the Riemann problem into
backward and forward rarefactions using \eq{1 simple wave} and \eq{3
  simple wave}, to get 
\begin{align*}
  u_m &= R(1,S)-R(\rho_m,S) = -R(\rho_m,S) \com{and}\\
  u_r - u_m &= R(1,S)-R(\rho_m,S) = -R(\rho_m,S),
\end{align*}
with no contact as $S_r=S_l=S$.  Adding, we must solve
\[
   u_r = -2\,R(\rho_m,S),
\]
and so if $u_r > - 2\,R(0+,S)$, no such $\rho_m$ can be found and a
vacuum is required to solve the Riemann problem; see
also~\cite{smoller,Young p system}. 
\end{proof}

We now introduce an easily checkable condition which implies
\eq{vacuum condition}, so is a sufficient condition for existence of
Riemann solutions with vacuum.  This \textbf{pressure near vacuum
  condition} describes the rate at which $\hat p \to 0$: for some
value of $S$, there exist positive numbers $\varepsilon_0$,
$\alpha_0$ and $M_0$, such that,
\beq
  \hat{p}_\rho(\rho,S)\leqslant M_0\rho^{\alpha_0}
   \com{whenever}  \rho\in(0,\varepsilon_0).
\label{pvc}
\eeq
Note that we require this condition at only one $S$: by continuity, we
would generally expect the condition to hold in an open set of $S$
values.  Note also that polytropic ideal gases satisfy \eq{pvc}, $\G
a_0$ being given by the adiabatic exponent $\gamma>0$.

\begin{theorem}
The energy condition \eq{energy condition} and pressure near vacuum
condition \eq{pvc} together imply that vacuums exist in the solution
of some Riemann problems.
\end{theorem}

\begin{proof}
According to \eq{p0}, the energy condition implies that 
$\hat p_\rho\to 0$ as $\G r\to0$, so \eq{pvc} makes sense.  It then
suffices by Lemma~\ref{lem:R} to show that, for $S$ given by \eq{pvc},
equation \eq{vacuum condition} is satisfied.  From \eq{Rdef}, for
$\rho<\varepsilon_0$, we have
\begin{align*}
   R(\varepsilon_0,S) - R(\rho,S) &=
    \int_\rho^{\varepsilon_0}\frac{\sqrt{\hat{p}_\rho(r,S)}}{r}\;dr\\
   &\leqslant \int_\rho^{\varepsilon_0} \sqrt{M_0}r^{-1+\alpha_0/2}\;dr\\
   &\leqslant \sqrt{M_0}\,\frac2{\alpha_0}\,\varepsilon_0^{\alpha_0/2},
\end{align*}
and taking the limit $\rho\to0$ gives the required lower bound for
$R(0+,S)$.
\end{proof}

%
%
\section{Gas dynamics without vacuum}

We now write down a class of constitutive laws for gases which
\emph{do not} admit vacuums in the solution of the Riemann problem.
These gases satisfy all the constraints of Section 2, but do not
satisfy the vacuum condition \eq{vacuum condition}.  By \eq{pvc},
there is a vacuum for any $\G c$-law gas ($p\sim \G r^{\G c}$) with
$\G c>1$, but no vacuum for an isothermal gas ($\G c=1$).  We thus
look for presssure laws between these cases, which restricts our
equation of state.

It is convenient to work with separable energies, that is, energies of
the form $E(\rho,S)=f(S)\,g(\rho)$.  Specifically, we consider
energies given by the expression
\beq
  E(\tau,S)=f(S)\int_{\ln(\tau+1)}^{+\infty} k^2(y)\;dy,
\label{example_e}
\eeq  
where $f$ and $k$ are $C^2$ functions satisfying certain conditions
described below.

We assume that $f(S)$ is $C^2$, positive, increasing and convex,
\beq
  f(S)>0,\quad f'(S)>0,\quad f''(S)\ge 0,
\label{fp}
\eeq
with limits
\beq
  \lim_{S\rightarrow\infty}f(S)=\infty \com{and}
  \lim_{S\rightarrow-\infty}f(S)=0.
\label{fl}
\eeq
We assume $k(y)$ is a $C^2$, convex, decreasing function defined on
$(0,\infty)$,
\beq
k(y)>0,\quad
k'(y)<0,\quad 
k''(y)>0,
\label{kp}
\eeq
subject to the growth limits
\beq
  \int_{1}^{+\infty} k^2(y)\;dy<+\infty \com{and}
  \int_{1}^{+\infty} k(y)\;dy={+\infty},
\label{k}
\eeq
describing growth near vacuum, $\tau\to\infty$, and limits
\beq
  k(y) \to \infty \com{and} \frac{k'(y)}{k(y)} \to -\infty
  \com{as} y\to 0,
\label{kto0}
\eeq
describing infinite density.

\begin{theorem}
\label{thm:uniq}
The above conditions imply that for gases having energy
\eq{example_e}, the Riemann problem has a unique global solution which
does not admit the vacuum state.  Moreover, such gases satisfy all
constraints of Section 2 with the exception of \eq{E cvx}.  
\end{theorem}

\begin{proof}
By \eq{example_e} and \eq{e_s_e_tau}, we have
\beq
  p(\tau,S)=\frac{1}{z}k^2(\ln z)f(S),
\label{pts}
\eeq
where we have set $z=\tau+1$.  It is easy to check that the standard
thermodynamic constraints, stability of matter, and thermodynamic
stability \eq{p tau} are satisfied.  The energy condition follows from
the first equation in \eq{k}.

We now check that the vacuum condition \eq{vacuum condition} fails for
this equation of state.  We calculate
\[
  p_\tau(\tau,S)=-\frac{k^2(\ln z)-2k(\ln z)k'(\ln z)}{z^2}f(S)<0,
\]
and, since $\rho=1/\tau$, we write
\[
  \hat p_\rho(\rho,S) = -p_\tau(\tau,S)\,\tau^2,
\]
and so, by \eq{Rdef},
\begin{align}
  R(\rho,S) &= -\int_{1}^{1/\rho}\sqrt{-p_\tau(\tau,S)}\;d\tau \nn\\
  &= -\sqrt{f(S)}\;\int_{\ln2}^{\ln(1+1/\rho)} \sqrt{k^2(w)-2k(w)k'(w)}\;dw,
\label{R}
\end{align}
where $w = \ln z=\ln(1+\tau)$.  It now follows from \eq{kp}, \eq{k} that
\[
  R(\rho,S) < -\sqrt{f(S)}\;\int_{\ln2}^{\ln(1+1/\rho)} k(w)\;dw \to -\infty
\]
as $\rho\to0$ for all $S$, so that \eq{vacuum condition} is never
satisfied.

Existence and uniqueness of Riemann solutions with arbitrary data now
follows from Smith's medium condition \eq{medium}, see~\cite{smith}.
Smith has an extra assumption, namely, he requires
\beq
  \lim_{\tau\rightarrow0^+}E(p,\tau)=0,\label{e limte tau 0}
\eeq
when $E$ is described as $E=E(p,\tau)$.  Because our energy is
separable, we have
\beq
  E(p,\tau)=p\frac{z\int_{\ln z}^{\infty}k^2(y)dy}{k^2(\ln z)}.
\label{Ep}
\eeq
Since $z\to1$ as $\tau\to0$, for small $\tau$, we write
\[
  E(p,\tau)=p\,z\,
  \frac{\int_{\ln z}^{1}k^2(y)dy+\int_{1}^{\infty}k^2(y)dy}{k^2(\ln z)}.
\]
By our assumptions on $k$, \eq{e limte tau 0} follows if we show
that the limit
\[
   \frac{\int_{\varepsilon}^{1}k^2(y)dy}{k^2(\varepsilon)} \to 0
  \com{as}\varepsilon\to0.
\]
This in turn follows from \eq{kto0} and l\'{}Hospital's rule.

To check the medium condition \eq{medium}, we use \eq{Ep} to write
\[
  p(\tau,E)=E\;\frac{k^2(\ln z)}{z\int_{\ln z}^\infty k^2(y)\;dy},
\]
where $z=\tau+1$.  The medium condition \eq{medium} is
\[
  \frac{2k(\ln z )k'(\ln z )\int_{\ln
 z }^\infty{k^2(y)\;dy}-k^2(\ln z )\int_{\ln
 z }^\infty{k^2(y)\;dy}+k^4(\ln  z )}{( z \int_{\ln
 z }^\infty{k^2(y)\;dy})^2}E<\frac{p^2}{2E},
\]
which simplifies to
\[
  N(z) := \frac{1}{2}k^3(z)+2k'(z)\int_z^\infty{k^2(y)\;dy}
  -k(z)\int_z^\infty{k^2(y)\;dy}<0,
\]
which must hold for all $z>1$.

By \eq{k}, $k(z)\to 0$ as $z\to\infty$, which in turn implies that
$N(z)\to0$ as $z\to\infty$. Furthermore, by (\ref{kp}), 
\[
  N'(z) = -\frac{1}{2}k'k^2+2k''\int_z^\infty{k^2}-k'\int_z^\infty{k^2}+k^3>0,
\]
so $N$ is increasing with limit 0, and thus $N(z)<0$ for all $z$.
\end{proof}

We remark that our conditions do not suffice to prove the convexity of
$E(\tau,S)$, namely the first condition of \eq{E cvx},
\[
  E_{\tau\tau}\cdot E_{SS}\geqslant E^2_{\tau S}.
\]
This condition is easily seen to be implied by the dual assumptions
that $f$ is log-convex,
\beq
   f(S)\,f{''}(S)\geqslant {f{'}}^2(S),
\label{flc}
\eeq
and $k$ satisfies the condition
\beq
  \big( k(z)- 2 k'(z)\big)\int_{z}^{\infty} k^2(y)\;dy\ge k^3(z).
\label{kcvx}
\eeq
for all $z$.

%
%
\section{Concrete example}

We now present the simplest equation of state which does not allow for
a vacuum in the solution of the Riemann problem.  Ideal polytropic
($\gamma$-law) gases with adiabatic constant $\G c>1$ admit vacuums,
while an isothermal gas ($\G c=1$) does not, but the isothermal gas
does not satisfy the finite energy stability of matter condition.  We
thus consider equations of state that fall between these two cases.
Since \eq{pvc} implies existence of vacuums, this restricts the
equation of state to those for which \eq{pvc} fails.

This class of gases we present here are polytropic, satisfying
\[
  E=c_\tau\;T,
\]
but do not satisfy the ideal gas law $p\,\tau = R\,T$.  For a
polytropic gas of the form \eq{example_e}, \eq{e_s_e_tau} implies that
we must have
\[
  f(S) = K_0\,e^{S/c_\tau},
\]
and this trivally satisfies conditions \eq{fp}, \eq{fl} and \eq{flc}.

We choose $k=k(z)$ as simple as possible so that properties \eq{kp}
and \eq{k} hold, namely
\beq
  k(z) = z^{-\G s}, \com{for some}  \G s\in(\half,1].
\label{kpower}
\eeq
It is then easy to check that \eq{kp}, \eq{k}, \eq{kto0} and \eq{kcvx}
hold.

With these choices, \eq{example_e} becomes
\beq
   E=\frac{K_0\,e^{S/c_\tau}}{2\sigma-1}(\ln(\tau+1))^{1-2\sigma},
\label{e example}
\eeq  
and \eq{pts} becomes
\beq
  p=\frac{K_0\,e^{S/c_\tau}}{\tau+1}(\ln(\tau+1))^{-2\sigma}.
\label{p example}
\eeq 

It follows from Theorem~\ref{thm:uniq} that the Riemann problem has a
unique solution without vacuum.  In fact, Smith's strong condition
holds, namely
\[
  \frac{\partial}{\partial{\tau}}E(\tau,p)>0.
\]
This is easily seen by eliminating $e^{S/c_\tau}$ from \eq{e example}
and \eq{p example}, to get
\[
  E(\tau,p) = p\;\frac{(\tau+1)\,\ln(\tau+1)}{2\sigma-1},
\]
and differentiating.

Although we have an abstract proof of existence and uniqueness
of Riemann solutions, we find it instructive to prove this directly.

\begin{lemma}
There is a unique solution to the Riemann problem for \eq{euler} with
equation of state \eq{e example}, which does \emph{not} include the
vacuum, for arbitrary Riemann data.
\end{lemma}

\begin{proof}
Using \eq{R} and simplifying, we have
\beq
  R(\tau,S) = - \sqrt{K_0}\;e^{S/2c_\tau}\;
    \int_{\ln2}^{\ln(\tau+1)}w^{-\sigma}\,
	\sqrt{1+\TS{\frac{2\sigma}{w}}}\;dw.
\label{Rt}
\eeq

Following \cite{chen}, we now make convenient changes of variables.
First, set
\[
  \phi:=\ln(\tau+1)>0,\com{so that} \tau+1=e^{\phi},
\]
and define
\[
  h := e^{-S/2c_\tau}\;R(\tau,S) 
    = -\sqrt{K_0}\;\int_{\ln2}^\phi w^{-\sigma}\,
	\sqrt{1+\TS{\frac{2\sigma}{w}}}\;dw.
\]
It is clear that $\rho$, $\tau$, $\phi$ and $h$ are equivalent
coordinates, with $\phi$ and $\tau$ decreasing and $h$ increasing as
functions of $\rho$, and since $\sigma\in(\frac12,1]$, we have
\beq
  \lim_{\phi\to\infty}h = - \infty \com{and}
  \lim_{\phi\to 0}h = \infty.
\label{philim}
\eeq
Next, define
\[
m := {\phi}^{-\sigma}e^{S/2c_\tau}
   = {(\ln(\tau+1))}^{-\sigma}e^{S/2c_\tau},
\]
so that $(\phi,m)$ can be used in place of $(\tau,S)$ or $(\rho,S)$.

It follows from \eq{e example} and \eq{p example} that
\beq 
  E = \frac{K_0}{2\sigma-1}\,\phi\,m^2  \com{and}
  p = K_0\, e^{-\phi}\,m^2,
\label{Epm}
\eeq
while also
\[
  e^{S/2c_\tau} = m\,\phi^\sigma \com{and}
  R = h\,m\,\phi^\sigma.
\]

The simple wave curves \eq{1 simple wave}, \eq{3 simple wave} are
described by
\beq
  u_r-u_l=(h_a-h_b)m_a \phi_a^{\sigma}, \qquad
  \frac{m_b}{m_a}=\frac{\phi_a^{\sigma}}{\phi_b^{\sigma}},
\label{simple}
\eeq
where the subscripts denote the behind, ahead, right and left states
respectively, and $\phi = \phi(h)$.  For rarefaction
waves, the sound speed $c$ decreases from front to back, so these are
characterized by $h_a>h_b$.  Similarly, by \eq{contact}, a contact
discontinuity is described by
\beq 
  u_r=u_l,  \qquad
  \frac{m_r}{m_l}=e^{(\phi_r-\phi_l)/2}.
\label{mcontact}
\eeq

It remains to calculate the shock curves.  Using \eq{Epm} in
(\ref{HC}), we get
\[
  \frac{K_0}{2\sigma-1}({\phi_b}{m_b}^2-{\phi_a}{m_a}^2)
    +K_0\frac{e^{-\phi_b}{m_b}^2
    +e^{-\phi_a}{m_a}^2}{2}(e^{\phi_b}-e^{\phi_a})=0,
\]
which becomes
\beq
  \frac{m_b}{m_a}=
   \sqrt{\frac{\frac{1}{2\sigma-1}\phi_a+\frac{1}{2}-
    \frac{1}{2}e^{\phi_b-\phi_a}}{\frac{1}{2\sigma-1}\phi_b+\frac{1}{2}-
    \frac{1}{2}e^{\phi_a-\phi_b}}}=: f(\phi_a,\phi_b).
\label{f def}
\eeq 
Here, $\phi_a>\phi_b$, or equivalently $h_a<h_b$, since the sound
speed $c$ is greater behind the shock \cite{courant,lax}.
It is clear that the function $f$ in (\ref{f def}) makes sense only if
the function inside the square root is nonnegative.  To check that
this holds, consider the function
\[
  q(x,y):=\frac{\frac{1}{2\sigma-1}x+\frac{1}{2}-
  \frac{1}{2}e^{y-x}}{\frac{1}{2\sigma-1}y+\frac{1}{2}-
  \frac{1}{2}e^{x-y}},
\] 
for $x>y>0$.   Denote the numerator and denominator of $q$ by 
\begin{align*}
  qn(x,y)&:=\frac{1}{2\sigma-1}x+\frac{1}{2}- \frac{1}{2}e^{y-x}
\com{and}\\
  qd(x,y)&:=\frac{1}{2\sigma-1}y+\frac{1}{2}- \frac{1}{2}e^{x-y},
\end{align*}
respectively.  It is immediate that for $x>0$,
\[
  qn(x,x)>0,\quad qn(x,0)>0, \quad
  qd(x,x)>0,\quad qd(x,0)<0.
\]
Since $qd(x,y)$ is strictly increasing with respect to $y$, for each
$x$, $qd(x,y)=0$ has a unique solution $0<y=\varphi(x)<x$.  On the
other hand, $qn(x,y)$ is strictly decreasing with respect to $y$, so
$qn(x,y)>0$ for $x>y>0$.  Hence, for fixed $x$,
\[
  q(x,y)>0 \com{as long as} \varphi(x) < y < x,
\]
while also
\[
  \lim_{y\to\varphi(x)^+}q(x,y)=+\infty.
\]

It follows that, for fixed $\phi_a$, $0<\varphi(\phi_a)<\phi_a$ and
the shock curve is parameterized by 
$\phi_b\in(\varphi(\phi_a),\phi_a)$, with
\beq 
  \lim_{\phi_b\rightarrow\varphi(\phi_a)^+}f(\phi_a,\phi_b)=+\infty.
\label{f varphi}
\eeq
Moreover, $f(\phi_a,\phi_b)$ increases with respect to $\phi_a$ and
decreases with respect to $\phi_b$.  Since $qd(x,y)=qn(y,x)$, we also
have
\[
  f(\phi_a,\phi_b)=\frac{1}{f(\phi_b,\phi_a)}.
\] 

Next, in these coordinates, \eq{ueq} is
\begin{align}
  [u]&=-\sqrt{K_0}\sqrt{( e^{-\phi_b}{m_b}^2-
         e^{-\phi_a}{m_a}^2)(e^{\phi_a}-e^{\phi_b})} \nn\\
    &=-\sqrt{K_0}\sqrt{( e^{\phi_a-\phi_b}\frac{{m_b}^2}{{m_a}^2}- 1)
        (1-e^{\phi_b-\phi_a})}\;m_a.
\label{ushock}
\end{align}
Defining
\beq
  g(\phi_a,\phi_b):=-\sqrt{K_0}\sqrt{(
  e^{\phi_a-\phi_b}f^2(\phi_a,\phi_b)- 1) (1-e^{\phi_b-\phi_a})},
\label{gd}
\eeq
it is easy to check that 
\[
  g(\phi_a,\phi_b)m_a=g(\phi_b,\phi_a)m_b,
\]
provided \eq{f def} holds, and $g(\phi_a,\phi_b)$ decreases with
respect to $\phi_a$ and increases with respect to $\phi_b$.  By
(\ref{f varphi}), since $\phi_a>\phi_b$,
\beq
  \lim_{\phi_b\rightarrow
  \varphi(\phi_a)}g(\phi_a,\phi_b)=-\infty.
\label{g varphi}
\eeq

To express the composite wave curves, we define 
\begin{align*}
  G(h_a,h_b) &:= \begin{cases}
    g(\phi_a,\phi_b),&  h_a<h_b<h(\varphi(\phi_a)), \\
    (h_a-h_b)\phi_a^\sigma, & h_a\ge h_b,
  \end{cases}  \com{and} \\
  F(h_a,h_b) &:= \begin{cases}
    f(\phi_a,\phi_b),& h_a<h_b<h(\varphi(\phi_a)), \\
    \phi_a^\sigma/\phi_b^\sigma, & h_a\ge h_b,
  \end{cases}
\end{align*}
with $\phi=\phi(h)$, and $\varphi(x)$ is defined by $qd(x,\vp(x))=0$.
Both $F(h_a,h_b)$ and $G(h_a,h_b)$ are defined on the region
$\{h_b<\vp(h_a)\}$, with $F$ increasing and $G$ decreasing in $h_b$
for fixed $h_a$.  Moreover, by \eq{philim}, we have the limits
\beq
  \lim_{h_b\to-\infty}F(h_a,h_b)=0 \com{and}
  \lim_{h_b\to-\infty}G(h_a,h_b)=+\infty,
\label{FGlim}
\eeq
for any $h_a$ fixed, and by (\ref{f varphi}) and (\ref{g varphi}),
\beq
  \lim_{h_b\to h(\varphi(\phi_a))}F(h_a,h_b)=+\infty \com{and}
  \lim_{h_b\to h(\varphi(\phi_a))}G(h_a,h_b)=-\infty.
\label{FGlim2}
\eeq

Combining \eq{simple} and \eq{f def}, \eq{ushock}, we describe the
composite $1$- and $3$-wave curves by
\beq
   u_r-u_l=G(h_a,h_b)m_a,\qquad
   \frac{m_b}{m_a}=F(h_a,h_b),
\label{1 3 wave}
\eeq
while the $2$-waves are given by \eq{mcontact}, namely
\[
   u_r=u_l,\qquad
   \frac{m_r}{m_l}=e^\frac{{\phi_r}-{\phi_l}}{2}.
\]

We now solve the Riemann problem by resolving the intermediate states.
We use subscripts $L$, $1$, $2$ and $R$ to denote the left,
intermediate, and right states, respectively.  We set
\[
  A:=\frac{u_R-u_L}{m_L} \com{and}
  B:=\frac{m_R}{m_L},
\] 
and use (\ref{1 3 wave}) and (\ref{mcontact}) to describe the waves,
and eliminate $u$, to get the equations
\begin{align}
  G(h_L,h_1)+B\,G(h_R,h_2)&=A, 
\label{Geq}\com{and}\\
  F(h_L,h_1)\,e^{-\phi_1/2}&=B\, F(h_R,h_2)\,e^{-\phi_2/2},
\label{Feq}
\end{align}
and we must solve for $h_1$ and $h_2$.

For any fixed $h_0$, the function $F(h_0,h)\,e^{-\phi(h)/2}$ is increasing
in $h$, and by \eq{FGlim}, \eq{FGlim2}, it has range $(0,\infty)$.  Thus
there exists a unique (increasing and $C^2$) function $\Theta$ of
$h_2$, such that
\[
  h_1=\Theta(h_2;B,h_L,h_R)
\]
if and only if \eq{Feq} holds.  Substituting into \eq{Geq} gives the
equation
\[
  G(h_L,\Theta(B,h_L,h_R,h_2))+B\, G(h_R,h_2)=A,
\]
and we must solve for $h_2$.  Since $G(h_0,h)$ is decreasing with $h$,
and by \eq{FGlim}, \eq{FGlim2} has range $(-\infty,\infty)$, there is
a unique solution $h_2 < \vp(\phi_R)$ of this equation.  Note that we
also have $h_1 < \vp(\phi_L)$.  Finally, we use \eq{1 3 wave} to fully
resolve the intermediate states, and the construction of the Riemann
solution is complete.
\end{proof}

\end{document}